\documentclass[8pt,reqno]{amsart}

\usepackage{amscd}
\usepackage[dvips]{graphics}

\def\beq{\begin{equation}}
\def\eeq{\end{equation}}
\def\ba{\begin{array}}
\def\ea{\end{array}}

\numberwithin{equation}{section}
\newtheorem{theorem}{Theorem}[section]

\newtheorem{corollary}[theorem]{\textbf{Corollary}}
\newtheorem{proposition}[theorem]{\textbf{Proposition}}

\theoremstyle{remark}
\newtheorem{remark}[theorem]{\textbf{Remark}}
\theoremstyle{plain}
\makeatletter
    \newcommand{\rmnum}[1]{\romannumeral #1}
    \newcommand{\Rmnum}[1]{\expandafter\@slowromancap\romannumeral #1@}
\makeatother

\begin{document}
\title{Differentiability of the $n$-Variable Function Deduced by the Differentiability of the $n-1$-Variable Function}
\author[Z. Ye, Q. Guo]{Zhenglin Ye, Qianqiao Guo}

\address{Zhenglin Ye, School of Mathematics and Statistics,  Northwestern Polytechnical University, Xi'an, 710129, Shaanxi, China}
\email{yezhenln@nwpu.edu.cn}

\address{Qianqiao Guo,  School of Mathematics and Statistics,  Northwestern Polytechnical University, Xi'an, 710129, Shaanxi, China}
\email{gqianqiao@nwpu.edu.cn}

\noindent
\begin{abstract}
In this paper, some sufficient conditions for the differentiability of the $n$-variable real-valued function are obtained, which are given based on the differentiability of the $n-1$-variable real-valued function and are weaker than classical conditions. 
\end{abstract}
\keywords{Sufficient Condition,  Differentiability, $n$-Variable Function}
\subjclass[2010]{26B05}

\maketitle


\section{Introduction}
\label{Section 1}

Consider a real-valued function $f: U(a)\to \mathbb{R}$ defined in a neighbourhood $U(a)\subset \mathbb{R}^n, n\ge2$ of the point $a:=(a_1, a_2, \cdots, a_n)$.

A classical sufficient condition for the differentiability of $f$ at the point $a$ states as (see for example \cite{Zor04}):

\begin{proposition}\label{Cl-Co}
If the function $f$ has all partial derivatives $f^\prime_{x_1}, \cdots, f^\prime_{x_n}$ at each point of the neighbourhood $U(a)$ and they are continuous at $a$, then $f$ is differentiable at $a$.  
\end{proposition}

In this paper, we show a different sufficient condition for the differentiability of $f(x_1,x_2,\cdots,x_n)$ at the point $a$, which are weaker than that of Proposition \ref{Cl-Co}.

 The main result is stated below. 
\begin{theorem}\label{main results-1}
Assume the function $f$ has all partial derivatives at the point $a$, and has partial derivatives $f^\prime_{x_{i_0}}(x_1,\cdots,x_{i_0-1},a_{i_0},x_{i_0+1},\cdots,x_n) $ at some neighbourhood of the point $\widehat{a}_{i_0}:=(a_1,\cdots,a_{i_0-1},\wedge,a_{i_0+1},\cdots,a_n)\in \mathbb{R}^{n-1}$, for fixed $1\le i_0\le n$. Then the function $f$ is differentiable at $a$ if and only if

(\rmnum1) the $n-1$-variable function $f(x_1,\cdots,x_{i_0-1},a_{i_0},x_{{i_0}+1},\cdots,x_n)$ is differentiable at the point $\widehat{a}_{i_0}$;

(\rmnum2) the $n-1$-variable function $f^\prime_{x_{i_0}}(x_1,\cdots,x_{{i_0}-1},a_{i_0},x_{{i_0}+1},\cdots,x_n) $ is continuous at the point $\widehat{a}_{i_0}$.
\end{theorem}
\begin{remark}
(1) The interesting point of Theorem \ref{main results-1} is that the differentiability of $f(x_1,x_2,\cdots,x_n)$ at the point $a$ can be deduced by the differentiability of the $n-1$-variable real-valued function $f(x_1,x_2,\cdots,x_{i_0-1},a_{i_0},x_{i_0+1},\cdots,x_n) (1\le i_0\le n)$ with some additional assumptions.

(2) Notice that if $f^\prime_{x_k}(x_1,\cdots,x_n)$ is continuous at $a$, then the conditions (\rmnum1) and (\rmnum2) hold. But we cannot deduce the continuity of $f^\prime_{x_k}(x_1,\cdots,x_n)$ at $a$ for any $1\le k\le n$ by (\rmnum1) and (\rmnum2). Thus, the conditions of Theorem \ref{main results-1} are weaker than that of Proposition \ref{Cl-Co}. 
\end{remark}

By inducement, it is easy to see that the following holds.
\begin{theorem}\label{main results-2}
Assume the function $f$ has all partial derivatives at the point $a$, and has partial derivatives $f^\prime_{x_k}(a_1,\cdots,a_k,x_{k+1},\cdots,x_n)$ at some neighbourhood of the point $\widehat{a}^{n-k}:=(a_{k+1},\cdots,a_n)\in \mathbb{R}^{n-k} (k=1,\cdots,n-1)$. Then the function $f$ is differentiable at $a$ if and only if

(1) the $n-1$-variable function $f^\prime_{x_1}(a_1,x_2,\cdots,x_n)$ is continuous at the point $\widehat{a}^{n-1}$;

(2) the $n-2$-variable function  $f^\prime_{x_2}(a_1,a_2,x_3,\cdots,x_n)$ is continuous at the point $\widehat{a}^{n-2}$;\\

$\cdots$\\

(n-1) the $1$-variable function  $f^\prime_{x_{n-1}}(a_1,\cdots,a_{n-1},x_n)$ is continuous at the point $\widehat{a}^{n-(n-1)}=a_n$. 
\end{theorem}

As an example, the case $n=3$ is stated as follows.
\begin{corollary}\label{3-dimensionalcase}
Assume the function $f(x_1,x_2,x_3)$ has all partial derivatives at the point $a$, and has partial derivative $f^\prime_{x_1}(a_1, x_2, x_3)$ at some neighbourhood of the point $(a_2,a_3)$,
and has partial derivative $f^\prime_{x_2}(a_1, a_2, x_3)$ at some neighbourhood of the point $a_3$. Then the function $f$ is differentiable at $a$ if and only if  $f^\prime_{x_1}(a_1, x_2, x_3)$ is continuous at the point $(a_2,a_3)$,
and $f^\prime_{x_2}(a_1, a_2, x_3)$ is continuous at the point $a_3$.
\end{corollary}

\begin{remark}
The functions $f^\prime_{x_1}(a_1, x_2, x_3)$ and  $f^\prime_{x_2}(a_1, a_2, x_3)$ above  can also be replaced by, for example, $f^\prime_{x_2}(x_1, a_2, x_3)$ and $f^\prime_{x_3}(x_1, a_2, a_3)$, or $f^\prime_{x_3}(x_1, x_2, a_3)$ and $f^\prime_{x_1}(a_1, x_2, a_3)$. 
\end{remark}
\section{Proofs of the main results}
\label{Section 2}

\begin{proof}[Proof of Theorem \ref{main results-1}]
   
(\Rmnum 1) The ``if" part.

Denote
\begin{eqnarray}\label{equ2-1}
&&\Delta R_n\\\nonumber
&:=&f(a_1+\Delta x_1, a_2+\Delta x_2, \cdots, a_n+\Delta x_n)-f(a_1, a_2, \cdots, a_n)-\sum\limits_{k=1}^n f^\prime_{x_k}(a_1, a_2, \cdots, a_n)\Delta x_k,\\\label{equ2-2}
&&\Delta R_{n-1}^{i_0}\\\nonumber
&:=&f(a_1+\Delta x_1, \cdots,a_{{i_0}-1}+\Delta x_{{i_0}-1}, a_{i_0}, a_{{i_0}+1}+\Delta x_{{i_0}+1}, \cdots, a_n+\Delta x_n)-f(a_1, a_2, \cdots, a_n)\\\nonumber
&&-\sum\limits_{k=1,k\neq {i_0}}^n f^\prime_{x_k}(a_1, a_2, \cdots, a_n)\Delta x_k
\end{eqnarray}
and $$\rho_n:= \Bigg(\sum\limits_{k=1}^n (\Delta x_k)^2 \Bigg)^{\frac{1}{2}},\ \ \rho_{n-1}^{i_0}:= \Bigg(\sum\limits_{k=1,k\neq {i_0}}^n (\Delta x_k)^2 \Bigg)^{\frac{1}{2}}.$$
Since $f$ is differentiable at $a$, we have
$$\lim\limits_{\rho_n \to 0}\frac{\Delta R_n}{\rho_n}=0.$$
Let $\Delta x_{i_0}=0$. Then it is easy to see that 
\begin{equation}\label{equ2-3}
\lim\limits_{\rho_{n-1}^{i_0} \to 0}\frac{\Delta R_{n-1}^{i_0}}{\rho_{n-1}^{i_0}}=0,
\end{equation}
which implies that the $n-1$-variable function $f(x_1,\cdots,x_{{i_0}-1},a_{i_0},x_{{i_0}+1},\cdots,x_n)$ is differentiable at the point $\widehat{a}_{i_0}$.

On the other hand, denote
\begin{eqnarray*}
&&\Delta f^\prime_{x_{i_0}}(a_1,\cdots,a_{{i_0}-1},\wedge,a_{{i_0}+1},\cdots,a_n)\\
&:=&f^\prime_{x_{i_0}}(a_1+\Delta x_1, \cdots,a_{{i_0}-1}+\Delta x_{{i_0}-1}, a_{i_0}, a_{{i_0}+1}+\Delta x_{{i_0}+1} \cdots, a_n+\Delta x_n)-f^\prime_{x_{i_0}}(a_1, \cdots, a_n).
\end{eqnarray*}   
Since $f$ has all derivatives at $a$, by using \eqref{equ2-1} and  \eqref{equ2-2}, we have 
\begin{eqnarray}\nonumber
&&\Delta R_n\\\nonumber
&=&f(a_1+\Delta x_1, \cdots, a_{{i_0}-1}+\Delta x_{{i_0}-1}, a_{i_0}, a_{{i_0}+1}+\Delta x_{{i_0}+1}, \cdots, a_n+\Delta x_n)\\\nonumber
&&+f^\prime_{x_{i_0}}(a_1+\Delta x_1, \cdots, a_{{i_0}-1}+\Delta x_{{i_0}-1}, a_{i_0}, a_{{i_0}+1}+\Delta x_{{i_0}+1}, \cdots, a_n+\Delta x_n)\Delta x_{i_0}+\alpha_{i_0} \Delta x_{i_0}\\\nonumber
&&-f(a_1, a_2, \cdots, a_n)-\sum\limits_{k=1}^n f^\prime_{x_k}(a_1, a_2, \cdots, a_n)\Delta x_k,\\\label{equ2-31}
&=&\Delta R_{n-1}^{i_0}+\Delta f^\prime_{x_{i_0}}(a_1,\cdots,a_{{i_0}-1},\wedge,a_{{i_0}+1},\cdots,a_n)\Delta x_{i_0}+\alpha_{i_0} \Delta x_{i_0},
\end{eqnarray}
where $\alpha_{i_0}\to0$ as $\Delta x_{i_0}\to 0$. 
Take $\Delta x_{i_0}=\rho_{n-1}^{i_0}$. Then $\rho_n=2^{\frac{1}{2}}\rho_{n-1}^{i_0}$. Noticing that $f$ is differentiable at $a$,  by using \eqref{equ2-3} and \eqref{equ2-31} we have
\begin{eqnarray*}
0&=&\lim\limits_{\rho_n \to 0}\frac{\Delta R_n}{\rho_n}\\\nonumber
&=&\lim\limits_{\rho_n \to 0}\Bigg(\frac{\Delta R_{n-1}^{i_0}}{\rho_n}+\Delta f^\prime_{x_{i_0}}(a_1,\cdots,a_{{i_0}-1},\wedge,a_{{i_0}+1},\cdots,a_n)\frac{\Delta x_{i_0}}{\rho_n}+\alpha_{i_0} \frac{\Delta x_{i_0}}{\rho_n}\Bigg)\\\nonumber
&=&2^{-\frac{1}{2}}\lim\limits_{\rho_{n-1}^{i_0} \to 0}\Delta f^\prime_{x_{i_0}}(a_1,\cdots,a_{{i_0}-1},\wedge,a_{{i_0}+1},\cdots,a_n),
\end{eqnarray*}
which implies that the partial derivative $f^\prime_{x_{i_0}}(x_1,\cdots,x_{{i_0}-1},a_{i_0},x_{{i_0}+1},\cdots,x_n)$ is continuous at the point $\widehat{a}_{i_0}$.

(\Rmnum 2) The ``only if" part.

Again by using \eqref{equ2-31}, we have 
\begin{eqnarray}\label{equ2-4}
&&\lim\limits_{\rho_n \to 0}\frac{\Delta R_n}{\rho_n}\\\nonumber
&=&\lim\limits_{\rho_n \to 0}\Bigg(\frac{\Delta R_{n-1}^{i_0}}{\rho_n}+\Delta f^\prime_{x_{i_0}}(a_1,\cdots,a_{{i_0}-1},\wedge,a_{{i_0}+1},\cdots,a_n)\frac{\Delta x_{i_0}}{\rho_n}+\alpha_{i_0} \frac{\Delta x_{i_0}}{\rho_n}\Bigg),
\end{eqnarray}
where $\alpha_{i_0}\to0$ as $\Delta x_{i_0}\to 0$. Since the $n-1$-variable function $f(x_1,\cdots,x_{{i_0}-1},a_{i_0},x_{{i_0}+1},\cdots,x_n)$ is differentiable at the point $\widehat{a}_{i_0}$, then \eqref{equ2-3} holds. Thus
\begin{equation}\label{equ2-5}
\lim\limits_{\rho_n \to 0}\frac{\Delta R_{n-1}^{i_0}}{\rho_n}=\lim\limits_{\rho_{n-1}^{i_0} \to 0}\frac{\Delta R_{n-1}^{i_0}}{\rho_{n-1}^{i_0}}\cdot \frac{\rho_{n-1}^{i_0}}{\rho_n}=0.
\end{equation}
On the other hand, the partial derivative $f^\prime_{x_{i_0}}(x_1,\cdots,x_{{i_0}-1},a_{i_0},x_{{i_0}+1},\cdots,x_n)$ is continuous at the point $\widehat{a}_{i_0}$. Then
\begin{equation}\label{equ2-6}
\lim\limits_{\rho_{n-1}^{i_0} \to 0}\Delta f^\prime_{x_{i_0}}(a_1,\cdots,a_{{i_0}-1},\wedge,a_{{i_0}+1},\cdots,a_n)=0.
\end{equation}
Therefore by using \eqref{equ2-4}, \eqref{equ2-5} and \eqref{equ2-6}, we conclude that   
\begin{equation*}
\lim\limits_{\rho_n \to 0}\frac{\Delta R_n}{\rho_n}=0.
\end{equation*}
That is,  the function $f$ is differentiable at $a$.   
\end{proof}

\begin{proof}[Proof of Theorem \ref{main results-2}]
By using Theorem \ref{main results-1}, it is easy to end the proof. We omit it here. 
\end{proof}

\end{document}